\documentclass[11pt]{article}
\usepackage{amsmath, amssymb, amsthm}
\usepackage{verbatim}
\usepackage{multicol}
\usepackage{hyperref}
\usepackage{enumerate}
\usepackage{tikz}

\date{}
\title{\vspace{-1.2cm}Rainbow saturation and graph capacities}
\author{
D\'aniel Kor\'andi \thanks{Institute of Mathematics, EPFL, Lausanne, Switzerland. Research supported in part by SNSF grants 200020-162884 and 200021-175977. Email: daniel.korandi@epfl.ch.}
}

\theoremstyle{plain}
\newtheorem{THM}{Theorem}[section]
\newtheorem*{THM*}{Theorem}

\newtheorem{COR}[THM]{Corollary}

\newtheorem{CONJ}[THM]{Conjecture}

\theoremstyle{definition}

\newcommand{\subs}{\subseteq}
\newcommand{\eps}{\varepsilon}

\newcommand{\calG}{\mathcal{G}}
\newcommand{\calT}{\mathcal{T}}

\DeclareMathOperator{\rsat}{rsat}
\begin{document}
\maketitle

\begin{abstract}
The $t$-colored rainbow saturation number $\rsat_t(n,F)$ is the minimum size of a $t$-edge-colored graph on $n$ vertices that contains no rainbow copy of $F$, but the addition of any missing edge in any color creates such a rainbow copy. Barrus, Ferrara, Vandenbussche and Wenger conjectured that $\rsat_t(n,K_s) = \Theta(n\log n)$ for every $s\ge 3$ and $t\ge \binom{s}{2}$. In this short note we prove the conjecture in a strong sense, asymptotically determining the rainbow saturation number for triangles. 
Our lower bound is probabilistic in spirit, the upper bound is based on the Shannon capacity of a certain family of cliques.
\end{abstract}

\section{Introduction}

A graph $G$ is called \emph{$F$-saturated} if it is a maximal $F$-free graph. The classic saturation problem, first studied by Zykov \cite{Z49} and Erd\H{o}s, Hajnal and Moon \cite{EHM64}, asks for the minimum number of edges in an $F$-saturated graph (as opposed to the Tur\'an problem, which asks for the maximum number of edges in such a graph). A rainbow analog of this problem was recently introduced by Barrus, Ferrara, Vandenbussche and Wenger \cite{BFVW17}, where a $t$-edge-colored graph is defined to be \emph{rainbow $F$-saturated} if it contains no rainbow copy of $F$ (i.e., a copy of $F$ where all edges have different colors), but the addition of any missing edge in any color creates such a rainbow copy. Then the $t$-colored rainbow saturation number $\rsat_t(n,F)$ is the minimum size of a $t$-edge-colored rainbow $F$-saturated graph.

Among other results, Barrus et al. showed that $\Omega\left(\frac{n\log n}{\log\log n}\right) \le \rsat_t(n,K_s) \le O(n\log n)$ and conjectured that their upper bound is of the right order of magnitude:
\begin{CONJ}[\cite{BFVW17}] \label{conj}
For $s\ge 3$ and $t\ge \binom{s}{2}$, $\rsat_t(n,K_s) = \Theta(n\log n)$.
\end{CONJ}

Here we prove this conjecture in a strong sense: we give a lower bound that is asymptotically tight for triangles. 

\begin{THM} \label{thm:main}
For $s\ge 3$ and $t\ge \binom{s}{2}$, we have
	\[ \rsat_t(n,K_s) \ge \frac{t(1+o(1))}{(t-s+2)\log(t-s+2)}n\log n \]
with equality for $s=3$.
\end{THM}

We should point out that Conjecture \ref{conj} was independently verified by Gir\~{a}o, Lewis and Popielarz \cite{GLP} and by Ferrara et al. \cite{FJL}, but with somewhat weaker bounds. In fact, our result proves a conjecture in \cite{GLP}, establishing the stronger estimate $\rsat_t(n,K_s)= \Theta_s(\frac{n\log n}{\log t})$ with their upper bound.

Our lower bound is probabilistic in spirit, using ideas of Katona and Szemer\'edi \cite{KS67}, and F\"uredi, Horak, Pareek and Zhu \cite{FHPZ98} (similar techniques were used in \cite{KLSS99,BS07,KS17}). The upper bound for $s=3$ is based on the following theorem that follows from a strong information-theoretic result of Gargano, K\"orner and Vaccaro \cite{GKV94} on the Shannon capacities of graph families.

\begin{THM} \label{thm:capacity}
For every $t\ge 3$, there is a set $X\subs [t]^k$ of $m=(t-1)^{(\frac{t-1}{t}-o(1))k}$ strings of length $k$ from alphabet $[t]=\{1,\dots,t\}$ such that for any $x,x'\in X$ and any $a\in[t]$, there is a position $i$ where $x(i)\ne x'(i)$ and $x(i),x'(i)\ne a$.
\end{THM}

In the next section we derive Theorem \ref{thm:capacity} from results about the Shannon capacity of graph families. This is followed by the proof of Theorem \ref{thm:main} in Section \ref{sec:main}.

\section{Graph capacities}

Let $\calG=\{G_1,\dots,G_r\}$ be a family of graphs on vertex set $[t]$. Let $N_k$ be the maximum size of a set $X\subs[t]^k$ of strings of length $k$ on alphabet $[t]$ such that for any two strings $x,x'\in X$ and any $G_j\in \calG$, there is a position $i_j\in [k]$ such that $x(i_j)x'(i_j)$ is an edge in $G_j$. The \emph{Shannon capacity} of the family $\calG$ is defined as $C(\calG)=\limsup_{k\to\infty}\frac{1}{k}\log N_k$  (see, e.g., \cite{S,CK}\footnote{The usual definition is with binary logarithm, but the base of our logarithms is unimportant for our purposes.}). When $\calG=\{G\}$, we  simply write $C(G)$ for $C(\calG)$.

We need an analogous definition for strings where the occurrences of each $a\in[t]$ are proportional to some probability measure $P$ on $[t]$. So let $\calT^k(P,\eps)$ be the set of all strings $x\in[t]^k$ such that $|\frac{1}{k}\#\{i:x(i)=a\}-P(a)|<\eps$ for every $a\in[t]$, and let $M_{k,\eps}$ be the maximum size of a set $X\subs \calT^k(P,\eps)$ such that for every $x,x'\in X$ there is an $i$ with $x(i)x'(i)\in G$. The Shannon capacity within type $P$ is $C(G,P)=\lim_{\eps\to 0}\limsup_{k\to\infty}\frac{1}{k}\log M_{k,\eps}$. Using a clever construction, Gargano, K\"orner and Vaccaro \cite{GKV94} showed that $C(\calG)$ can be expressed in terms of the $C(G_j,P)$:
\begin{THM}[\cite{GKV94}] \label{thm:korner}
For a family of graphs $\calG=\{G_1,\dots,G_r\}$ on vertex set $[t]$, we have
\[C(\calG)=\max_P \min_{G_j\in \calG} C(G_j,P).\]
\end{THM}
In fact, they proved a more general result for \emph{Sperner capacities}, the analogous notion for directed graphs. What we need is a corollary that follows easily from this theorem using standard tools about graph entropy (see the survey of Simonyi \cite{S} for more information). Here we give a self-contained argument that goes along the lines of a proof by Gargano, K\"orner and Vaccaro \cite{GKV93} of the case $s=2$.
\begin{COR} \label{cor:main}
Let $2\le s\le t$ be an integer and let $\calG$ be the family of all $s$-cliques on $[t]$ (each with $t-s$ isolated vertices). Then $C(\calG)=\frac{s}{t}\log s$.
\end{COR}
\begin{proof} For the lower bound, we can take $P$ to be the uniform measure on $[t]$. Then by Theorem \ref{thm:korner}, it is enough to show that $C(G,P)\ge \frac{s}{t}\log s$ where $G$ is a clique on $[s]$ with isolated vertices $s+1,\dots,t$. Let $X_k\subs\calT^k(P,\frac{1}{k})$ be the set of all strings $x$ of length $k$ such that the first $\lfloor sk/t\rfloor$ letters of $x$ contain $\lfloor k/t\rfloor$ or $\lceil k/t\rceil$ instances of each $a\in[s]$, and $x(i)=b$ for every $s+1\le b\le t$ and $\frac{(b-1)k}{t}<i\le \frac{bk}{t}$. Then
\[C(G,P) \ge \lim_{k\to\infty} \frac{\log(X_k)}{k} =  \lim_{k\to\infty} \frac{1}{k}\log \frac{(\frac{sk}{t})!}{((\frac{k}{t})!)^s} =  \lim_{k\to\infty} \frac{1}{k} \log(s^{sk/t}) = \frac{s}{t}\log s.\]

For the upper bound, let $X\subs [t]^k$ be a maximum set of strings such that for any $x,x'\in X$ and for every $s$-clique $G\in \calG$, there is an $i\in [k]$ such that $x(i)x'(i)\in G$. We set $m=|X|$ to be this maximum. We may assume that $\{1,\dots,s\}$ are the $s$ least frequent elements appearing in the strings of $X$. Let $d_x$ be the number of elements in $x\in X$ that are not in $[s]$, so $\sum_{x\in X}d_x\ge \frac{t-s}{t}mk$, and let $X_x$ be the set of strings obtained from $x$ by replacing these elements arbitrarily with numbers from $[s]$. Then $|X_x|=s^{d_x}$, and $X_x, X_{x'}$ are disjoint for distinct $x,x'\in X$ because any string from $X_x$ will differ from any string in $X_{x'}$ at the position $i$ where $x(i)x'(i)$ is an edge of the clique on $[s]$. Then using Jensen's inequality we have
\[s^k \ge \sum_{x\in X} s^{d_x} \ge m\cdot s^{(\sum_{x\in X} d_x)/m} \ge m \cdot s^{\frac{(t-s)k}{t}},\]
and hence $m \le s^{sk/t}$, implying $C(\calG)\le \frac{1}{k}\log m \le \frac{s}{t}\log s$.
\end{proof}
Theorem \ref{thm:capacity} clearly follows from the case $s=t-1$. 

\section{Rainbow saturation} \label{sec:main}

\begin{proof}[Proof of Theorem \ref{thm:main}]
For the lower bound, suppose $H$ is a $t$-edge-colored rainbow $K_s$-saturated graph, and split its vertices into two parts: let $A=\{a_1,\dots,a_k\}$ be the set of vertices of degree at least $d= \log^3 n$, and $B$ be the rest. We may assume $|A|\le \frac{n}{\log n}$ (otherwise $H$ has at least $\frac{1}{2}n\log^2n$ edges), and thus $B$ contains $m\ge (1-\frac{1}{\log n})n$ vertices. Now let us define a string $x_v\subs[t+1]^k$ for every $v\in B$ that encodes the colors of the $A$-$B$ edges touching $v$ as follows: $x_v(i)$ is $t+1$ if $a_iv$ is not an edge in $H$, otherwise it is the color of $a_iv$.
	
Assume, without loss of generality, that $t-s+3,\dots,t$ are the $s-2$ most common colors among the $A$-$B$ edges. For $v\in B$, let $X_v\subs [t-s+2]^k$ be the set of strings obtained from $x_v$ by replacing each $t-s+3,\dots,t+1$ with an arbitrary number from $[t-s+2]$. Then if $d_v$ denotes the number of $A$-$B$ edges in $H$ touching $v$ and $d'_v$ denotes the number of such edges of colors $t-s+3,\dots,t$, then $|X_v|=(t-s+2)^{k-d_v+d'_v}$.
	
We claim that if $v,w\in B$ are non-adjacent with no common neighbor in $B$, then $X_v$ and $X_w$ have no string in common. Indeed, adding the edge $vw$ of color $t$ creates a rainbow $K_s$ with $s-2$ vertices in $A$. So there must be an $a_i$ such that $a_iv$ and $a_iw$ have different colors, also differing from $t-s+3,\dots,t$. But then all the strings in $X_v$ have the color of $a_iv$ as their $i$'th letter, and all the strings in $X_w$ have the color of $a_iw$ as their $i$'th letter, so $X_v$ and $X_w$ are disjoint.
	
Since vertices in $B$ have degree at most $d$, each $v\in B$ has at most $d^2$ vertices $w\in B$ that are either adjacent to $v$ or have a common neighbor with $v$ in $B$. So each string in $[t-s+2]^k$ can appear in no more than $d^2+1$ collections $X_w$, and hence we get
\begin{align*}
	(d^2+1)(t-s+2)^k &\ge \sum_{v\in B} |X_v| = \sum_{v\in B} (t-s+2)^{k-d_v+d'_v}  \\
	d^2+1 &\ge \sum_{v\in B} (t-s+2)^{d'_v-d_v} \ge m\cdot (t-s+2)^{ \frac{1}{m} (\sum_{v\in B}d'_v - \sum_{v\in B}d_v )}
\end{align*}
using Jensen's inequality.

Now $t-s+3,\dots,t$ were the $s-2$ most common colors, so we also have $\sum_{v\in B}d'_v \ge \frac{s-2}{t}\sum_{v\in B}d_v$ and thus $\sum_{v\in B}d'_v - \sum_{v\in B}d_v \ge  \frac{s-2-t}{t}\sum_{v\in B} d_v$. 
Taking logs, we obtain
\[ \sum_{v\in B} d_v \ge \frac{t}{t-s+2}m\left(\log_{t-s+2} m  - \log_{t-s+2} (d^2+1)\right). \]
As the left-hand side is a lower bound on the number of edges in $H$, this establishes the desired lower bound (using $d=\log^3 n$ and $m=n+o(n)$).

\medskip
For the upper bound in the case of triangles, let $k$ be large enough, and take a set $X$ of size $m$ as provided by Theorem \ref{thm:capacity}. Consider a $k$-by-$m$ complete bipartite graph $G_0$ with parts $A$ and $B$, where $A=\{a_1,\dots,a_k\}$, and $B$ corresponds to the strings in $X$. For every vertex $v\in B$, we look at the corresponding string $x\in X$, and color each edge $va_i$ by the color $x(i)$. $G_0$ is clearly (rainbow) triangle-free, and by the definition of $X$, adding an edge to $G_0$ between two vertices of $B$ in any color $a\in[t]$ creates a rainbow triangle.

Now let $G$ be a maximal rainbow triangle-free supergraph of $G_0$. Then $G$ is rainbow triangle-saturated by definition, and compared to $G_0$, it only has new edges induced by $A$, thus it has at most $km+\binom{k}{2}$ edges. Here $n=k+m$ and $k=\frac{t(1+o(1))}{(t-1)\log(t-1)}\log m$, implying the required upper bound.
\end{proof}

For $s>3$ our lower bound is probably not tight. It would be interesting to determine the asymptotics of $\rsat_t(n,K_s)$ for general $s$.

\bigskip
\noindent\textbf{Acknowledgements.}
I thank Shagnik Das for finding \cite{GKV93} for me, and G\'abor Simonyi for some clarifications about capacities.


\begin{thebibliography}{99}


\bibitem{BFVW17} M.~D.~Barrus, M.~Ferrara, J.~Vandenbussche and P.~S.~Wenger,
\newblock{Colored saturation parameters for rainbow subgraphs},
\newblock{\em J. Graph Theory}, \textbf{86} (2017), 375-386.

\bibitem{BS07} B.~Bollob\'as and A.~Scott,
\newblock{Separating systems and oriented graphs of diameter two},
\newblock{\em J. Combin. Theory Ser. B} \textbf{97} (2007), 193-203.

\bibitem {CK} I.~Csisz\'ar and J.~K\"orner,
\newblock{Information Theory}, 2nd edition, Cambridge University Press, 2011.

\bibitem{EHM64} P.~Erd\H{o}s, A.~Hajnal and J.W.~Moon,
\newblock{A problem in graph theory},
\newblock{\em Amer. Math. Monthly}, \textbf{71} (1964), 1107-1110.

\bibitem{FJL} M.~Ferrara, D.~Johnston, S. Loeb, F.~Pfender, A.~Schulte, H.~C.~Smith, E.~Sullivan, M.~Tait and C.~Tompkins,
\newblock{On edge-colored saturation problems},
arXiv:1712.00163 preprint

\bibitem{FHPZ98} Z.~F\"uredi, P.~Horak, C.~M.~Pareek and X.~Zhu,
\newblock{Minimal oriented graphs of diameter 2},
\newblock{\em Graphs Combin.} \textbf{14} (1998), 345-350.

\bibitem{GKV93} L.~Gargano, J.~K\"orner and U.~Vaccaro,
\newblock{Sperner capacities},
\newblock{\em Graphs Combin.}, \textbf{9} (1993), 31-46.

\bibitem{GKV94} L.~Gargano, J.~K\"orner and U.~Vaccaro,
\newblock{Capacities: from information theory to extremal set theory},
\newblock{\em J. Combin. Theory Ser. A}, \textbf{68} (1994), 296-316.

\bibitem{GLP} A.~Gir\~{a}o, D.~Lewis and K.~Popielarz,
\newblock{Rainbow saturation of graphs},
arXiv:1710.08025 preprint

\bibitem{KS67} G.~Katona and E.~Szemer\'edi,
\newblock{On a problem of graph theory},
\newblock{\em Studia Sci. Math. Hungar.} \textbf{2} (1967), 23-28.

\bibitem{KS17} D.~Kor\'andi and B.~Sudakov,
\newblock{Saturation in random graphs},
\newblock{\em Random Structures Algorithms} \textbf{51} (2017), 169-181.

\bibitem{KLSS99} A.~V.~Kostochka, T.~\L uczak, G.~Simonyi and E.~Sopena,
\newblock{On the minimum number of edges giving maximum oriented chromatic number}, in:
\newblock{Contemporary Trends in Discrete Mathematics},
\newblock{\em DIMACS Series in Discrete Mathematics and Theoretical Computer Science}, vol 49. (1999), 179-182.

\bibitem {S} G.~Simonyi,
\newblock{Perfect graphs and graph entropy. An updated survey}, in:
\newblock{Perfect Graphs}, Wiley (2001), 293-328.

\bibitem{Z49} A.~Zykov,
\newblock{On some properties of linear complexes (in Russian)},
\newblock{\em Mat. Sbornik N. S. } \textbf{24} (1949), 163-188.


\end{thebibliography}
\end{document}